\documentclass[12pt,leqno]{amsart}
\usepackage{amssymb,amsthm,amsmath,latexsym}
\newtheorem{theorem}{\sc Theorem}[section]

\newtheorem{lem}[theorem]{\sc Lemma}

 \newtheorem*{thmA}{Theorem A}
 \newtheorem*{thmB}{Theorem B}

 \DeclareMathOperator{\PSL}{PSL}
 \DeclareMathOperator{\FC}{FC}

 \DeclareMathOperator{\ASL}{ASL}

 \DeclareMathOperator{\Aut}{Aut} 
 \DeclareMathOperator{\Tor}{Tor}  
 \DeclareMathOperator{\cent}{Z}

\title[Automorphism Orbits]{Soluble Groups with few orbits under automorphisms}
\author[Bastos]{Raimundo  Bastos}
\address{Departamento de Matem\'atica, Universidade de Bras\'ilia,
Brasilia-DF, 70910-900 Brazil}
\email{(Bastos) bastos@mat.unb.br}
\email{(Dantas) alexcdan@gmail.br}
\email{(de Melo) emersonueg@hotmail.com}
\author[Dantas]{Alex C. Dantas}
\author[de Melo]{Emerson de Melo}
\thanks{This work was partially supported by FAPDF - Brazil.}
\begin{document}
\maketitle

\begin{abstract}
Let $G$ be a group. The orbits of the natural action of $\Aut(G)$ on $G$ are called ``automorphism orbits'' of $G$, and the number of automorphism orbits of $G$ is denoted by $\omega(G)$. We prove that if $G$ is a soluble group with finite rank such that  $\omega(G)< \infty$, then $G$ contains a torsion-free characteristic nilpotent subgroup $K$ such that $G = K \rtimes H$, where $H$ is a finite group. Moreover, we classify the mixed order soluble groups of finite rank such that $\omega(G)=3$.
\end{abstract}

%%% ----------------------------------------------------------------------
\maketitle
%%% ----------------------------------------------------------------------
%\tableofcontents

\section{Introduction}

Let $G$ be a group. The orbits of the natural action of $\Aut(G)$ on $G$ are called ``automorphism orbits'' of $G$, and the number of automorphism orbits of $G$ is denoted by $\omega(G)$. It is interesting to ask what can we
say about ``$G$'' only knowing $\omega(G)$. It is obvious that $\omega(G)=1$ if and only if $G=\{1\}$, and it is well known that if $G$ is a finite group then $\omega(G) = 2$ if and only if $G$ is elementary abelian. In \cite{LM}, T.\,J. Laffey and D. MacHale proved that if $G$ is a finite non-soluble group with $\omega(G) \leqslant 4$, then $G$ is isomorphic to $\PSL(2,\mathbb{F}_4)$. Later, M. Stroppel, in \cite{S1}, has shown that the only finite non-abelian simple groups $G$ with $\omega(G) \leq 5$ are the groups $\PSL(2,\mathbb{F}_q)$ with $q \in \{4,7,8,9\}$. In \cite{BDG}, the authors prove that if $G$ is a finite non-soluble group with $\omega(G) \leq 6$, then $G$ is isomorphic to one of $\PSL(2,\mathbb{F}_q)$ with $q \in \{4,7,8,9\}$, $\PSL(3,\mathbb{F}_4)$ or $\ASL(2,\mathbb{F}_4)$ (answering a question of M. Stroppel, cf. \cite[Problem 2.5]{S1}). 

Some aspects of automorphism orbits are also investigated for infinite groups.  M. Schwachh\"ofer and M. Stroppel in \cite[Lemma 1.1]{S2}, have shown that if $G$ is an abelian group with finitely many automorphism orbits, then $G = \Tor(G) \oplus D$, where $D$ is a characteristic torsion-free divisible subgroup of $G$ and $\Tor(G)$ is the set of all torsion elements in $G$. In \cite[Theorem A]{BD18}, the authors proved that if $G$ is a $\FC$-group with finitely many automorphism orbits, then the derived subgroup $G'$ is finite and $G$ admits a decomposition $G = \Tor(G) \times A$, where $A$ is a divisible characteristic subgroup of $\cent(G)$. For more details concerning automorphism orbits of groups see \cite{S1}. 

If $G$ is a group and $r$ is a positive integer, then $G$ is said to have finite
rank $r$ if each finitely generated subgroup of $G$ can be generated by $r$ or fewer elements and if $r$ is the least such integer. The next result can be viewed as a generalization of the above mentioned results from \cite{BD18} and \cite{S2}.  

\begin{thmA}
Let $G$ be a soluble group of finite rank. If $\omega(G) < \infty$, then $G$ has a torsion-free radicable nilpotent subgroup $K$ such that $G = K \rtimes H$, where $H$ is a finite subgroup.
\end{thmA}
We do not know whether the hypothesis that $G$ has finite rank is really needed in Theorem A. The proof that we present here uses this assumption in a very essential way. 

In \cite{LM}, T.\,J. Laffey and D. MacHale showed that $G$ is a finite group in which the order $|G|$ is not prime power and $\omega(G)=3$ if and only if $|G|=pq^{n}$, the Sylow $q$-subgroup $Q$ is a normal elementary abelian subgroup of $G$ and $P$ is a Sylow $p$-subgroup which acts fixed-point-freely on $Q$. See also \cite{MS1} for groups with $\omega(G) \leqslant 3$ (almost homogeneous groups). 

Recall that a group $G$ has mixed order if it contains non-trivial elements of finite order and also elements of infinite order. We obtain the following classification.  

\begin{thmB}
Let $G$ be a mixed order soluble group with finite rank. We have $\omega(G)=3$ if and only if $G=A \rtimes H$ where $|H|=p$ for some prime $p$, $H$ acts fixed-point-freely on $A$ and $A=\mathbb{Q}^n$ for some positive integer $n$.
\end{thmB}

\section{Proofs}

A well-known result in the context of extensions of finite groups, due to I. Schur,  states that if $G$ is a finite group and $N$ is a normal abelian subgroup with $(|N|,|G:N|)=1$, then there exists a complement $K$ of $N$. Recall that $K$ is a complement of (a normal subgroup) $N$ in $G$ if $N \cap K=1$ and $G = NK$. In particular, $G = N \rtimes K$. Now, we prove that Schur's theorem holds under a more general assumption that $G$ contains a divisible abelian subgroup of finite index. The proof presented here is adapted from the ideas of the finite case (cf. \cite[9.1.2]{Rob}).   

\begin{lem}\label{lem.abelian.sub}
Let $A$ be a divisible normal abelian subgroup of finite index of a group $G$. Then there exists a subgroup $H$ of $G$ such that $G = A \rtimes H$.
\end{lem}

\begin{proof}
Let $B = G/A$. From each coset $x$ in $B$ we choose a representative $t_{x}$, so that the set $T = \{t_{x} \mid x \in B\}$ is a transversal to $A$ in $G$. Since $t_{x}t_{y}A = t_{xy}A$, there is an element $c(x, y)$ of $A$ such that $t_{x}t_{y} = t_{xy}c(x, y)$. Then
$$(t_{x}t_{y})t_{z} = t_{xy}c(x, y)t_{z} = t_{xy}t_{z}c(x, y)^{z} = t_{xyz}c(xy, z)c(x, y)^{z},$$
$$t_{x}(t_{y}t_{z}) = t_{x}t_{yz}c(y, z) = t_{xyz}c(x, yz)c(y, z)$$
and $c(xy, z)c(x, y)^{z} = c(x, yz)c(y, z)$, for each $x, y \in B$. Consider the element $d(y) = \prod_{x \in B} c(x, y) \in A$. As $A$ is an abelian group $d(z)d(y)^{z} = d(yz)c(y, z)^{n}$, where $n = |B|$. We obtain that $d(yz) = d(y)^{z}d(z)c(y, z)^{-n}$. Now, since $A$ is a divisible group, there exists $e(y) \in A$ such that $e(y)^{n} = d(y)^{-1}$ for each $y \in B$. Hence
$$e(zy)^{-n} = (e(y)^{z}e(z)c(y, z))^{-n}.$$
Since $A$ is torsion-free, it follows that $e(zy) = (e(y)^{z}e(z)c(y, z))$. Define $s_{x} = t_{x}e(x)$, then
$$s_{y}s_{z} = t_{y}t_{z}e(y)^{z}e(z) = t_{yz}c(y, z)e(y)^{z}e(z) = t_{yz}e(yz) = s_{yz}.$$
Thus $x \mapsto s_{x}$ defines a homomorphism $\phi: B \rightarrow G$. Now $s_{x} = 1$ implies  that $t_{x} \in A$ and $x = A = 1_{B}$. From this we conclude that $H = B^{\phi}$ is the desired complement.
\end{proof}

Now, we consider groups with finitely many automorphism orbits with a characteristic torsion-free soluble subgroup of finite index (see also Schur-Zassenhaus Theorem \cite[9.1.2]{Rob}). 
\begin{lem}
\label{lem.soluble.sub} Let $n$ be a positive integer. Let $G$ be a group such that $\omega(G)< \infty$ and $A$ a torsion-free characteristic subgroup of $G$ with finite index $n$. If $A$ is soluble, then there exists a subgroup $H$ of $G$ such that $G = A \rtimes H$. 
\end{lem}

\begin{proof}
If $A$ is abelian, then the result is immediate by Lemma  \ref{lem.abelian.sub}. Assume that $A$ is non-abelian. Set $d$ the derived length of $A$. First we prove that $A/A^{(d-1)}$ is torsion-free. The subgroup $A^{(d-1)}$ is a torsion-free abelian divisible subgroup since $A$ is torsion-free and has finitely many automorphism orbits. If $A/A^{(d-1)}$ is not torsion-free, then we can find an element $a\in A$ such that $\langle a,A^{(d-1)}\rangle$ has a torsion-free divisible group of finite index. Then by  Lemma  \ref{lem.abelian.sub} the subgroup $\langle a,A^{(d-1)}\rangle$ has elements of finite order. That is a contradiction. Thus $A/A^{(d-1)}$ is torsion-free.

Now, we complete the proof arguing by induction on the derived length of $A$. Consider the quotient group $\bar{G} = G/A^{(d-1)}$. By induction we deduce that there exists a finite subgroup $\bar{B}$ of order $n$ in $\bar{G}$ such that $\bar{G} = \bar{A} \rtimes \bar{B}$. Set $B$ the inverse image of $\bar{B}$. Clearly $A^{(d-1)} \leq B$ and $A^{(d-1)}$ has finite index $n$ in $B$. Therefore, by Lemma  \ref{lem.abelian.sub} $B$ has a subgroup $H$ of order $n$ and so such a subgroup is a complement of $A$ in $G$. The result follows. 
\end{proof}

The following lemma is well-known. We supply the proof for the reader's convenience.

\begin{lem} \label{lem.abelian}
Let $G$ be an abelian group of finite rank. If $\omega(G)< \infty$, then the torsion subgroup $\Tor(G)$ is finite.  
\end{lem}

\begin{proof}
Since $G$ has finitely many automorphism orbits, it follows that the exponent $\exp(\Tor(G))$ is bounded. As $G$ has finite rank we have that $\Tor(G)$ is finitely generated. We deduce that $\Tor(G)$ is finite, which completes the proof.    
\end{proof}

The following result provides a description of radicable nilpotent groups of finite rank (see \cite[Theorem 5.3.6]{Rob2} for more details). Recall that a group $G$ is said to be radicable if each element is an $n$th power for every positive integer $n$.

\begin{lem} \label{lem.radicablenil}
Let $G$ be a soluble group with finite rank. Then the following are
equivalent:
\begin{enumerate}
    \item[(i)] $G$ has no proper subgroups of finite index;
    \item[(ii)] $G = G^m$ for all $m > 0$;
    \item[(iii)] $G$ is radicable and nilpotent.
\end{enumerate}

\end{lem}

We are now in a position to prove Theorem A. 

\begin{proof}[Proof of Theorem A]
We argue by induction on derived length of $G$. 

Assume that $G$ is abelian. By Schwachh\"ofer-Stroppel's result \cite{S2}, $G = D \oplus T$, where $D$ is characteristic torsion free divisible subgroup and $T$ is the torsion subgroup of $G$. By Lemma \ref{lem.abelian}, the torsion subgroup $T = \Tor(G)$ is a finite subgroup of $G$, the result follows.    

Now, we assume that $G$ is non-abelian. 
Set $d$ the derived length of $G$. Arguing as in the previous paragraph, we deduce that  $G^{(d-1)} = D_1 \oplus T_1$, where $D_1$ is a characteristic torsion-free divisible subgroup and $T_1$ is the torsion subgroup of $G^{(d-1)}$ and so, $T_1$ is finite. By induction $G/G^{(d-1)}$  has the desired decomposition. More precisely, $G^{(d-1)} = D_1 \oplus T_1$ and $G/G^{(d-1)} = \bar{A} \rtimes \bar{B}$ where $\bar{A}$ is torsion-free and $\bar{B}$ is finite. Note that $\bar{A}^n=\bar{A}$ for any positive integer $n$, since the quotient groups  $\bar{A}^{(i)}/\bar{A}^{(i+1)}$ are torsion-free divisible groups (we can use Lemma \ref{lem.soluble.sub} to conclude that each quotient is torsion-free).

Note that the centralizer $C_G(T_1)$ is a subgroup of finite index in $G$, because $G/C_G(T_1)$ embeds in the automorphism group of $T_1$ which has finite order. Let $A$ be the inverse image of $\bar{A}$. As $\bar{A}$ is torsion-free and $\bar{A}^n=\bar{A}$ for any positive integer $n$, we have $A \leqslant C_G(T_1)$. Thus $T_1\leq Z(A)$ and $\Tor(A)=T_1$. Set $K=A^e$, where $e=\exp(T_1)$. Then $K$ is torsion-free and has finite index in $G$. Therefore, by Lemma \ref{lem.soluble.sub}, there exists a finite subgroup $H$ such that $G = K \rtimes H$. According to Lemma \ref{lem.radicablenil}, we deduce that $K$ is a radicable nilpotent group (the subgroup $K$ has no proper subgroups of finite index). The proof is complete. 
\end{proof}

\noindent Now we will deal with Theorem B: {\it Let Let $G$ be a mixed order soluble group with finite rank. We have $\omega(G)=3$ if and only if $G=A \rtimes H$ where $A=\mathbb{Q}^n$ for some positive integer $n$, $|H|=p$ for some prime $p$ and $H$ acts fixed-point-freely on $A$.} 

\begin{proof}[Proof of Theorem B] First assume that $G$ is a mixed order soluble group of finite rank and have $\omega(G)=3$. By Theorem A, $G = A \rtimes H$ where $A$ is a torsion-free radicable nilpotent subgroup and $H$ is a finite group. Since $\omega(G)=3$ and $G$ has mixed order, it follows that $A$ must be abelian (so that $A=\mathbb{Q}^n$) and $H$ is an elementary abelian $p$-subgroup. On the other hand, since $A$ is characteristic, we deduce that all elements in $G\setminus A$ have order $p$ and then $H$ acts fixed-point-freely on $A$, so that $H$ is cyclic \cite[Lemma 2.2]{MS1}.

Conversely, suppose that $G=A \rtimes H$ where $A=\mathbb{Q}^n$, $|H|=p$ for some prime $p$, and $H$ acts fixed-point-freely on $A$. First we prove that all elements in $G \setminus A$ have order $p$. Set $h \in G$ such that $H = \langle h \rangle$. Since $h^p = 1$, we conclude that the minimal polynomials $m_{h} = m_{h}(x)$ of $T_{h}: \mathbb{Q}^n \to \mathbb{Q}^n$, given by $a\mapsto a^{h}$,
divides 
$$x^p - 1 = (x - 1)(x^{p-1} + x^{p-2} + \ldots + x + 1).$$
Consequently, $m_{h}(x) = x^{p-1} + x^{p-2} + \ldots + x + 1$, because $h$ acts fixed-point-freely on $A$. Thus using the identity $(xy)^n = x^n(y^{x^{n-1}}) \ldots y^{x}y$, we obtain  
$$(h^{j}a)^p = h^{jp}(aa^{h^{j}}\ldots a^{h^{(p-1)j}})= 1$$
for all $a \in \mathbb{Q}^{n}$. So all elements of $G \setminus A$ have order exactly $p$ and act fixed-point-freely on $A$. 

Now, let $b, c \in A \setminus \{1\}$ and $\alpha,\beta \in G \setminus A$. By Cyclic Decomposition Theorem, there exist $b_{1}, b_{2}, \ldots, b_{t}, c_{1}, c_{2}, \ldots, c_{t} \in A$ such that
$$\{b_{1}, b_{1}^{\alpha}, \ldots, b_{1}^{\alpha^{p-2}}, \ldots, b_{t}, b_{t}^{\alpha}, \ldots, b_{t}^{\alpha^{p-2}}\}$$ and 
$$\{c_{1}, c_{1}^{\beta}, \ldots, c_{1}^{\beta^{p-2}}, \ldots, c_{t}, c_{t}^{\beta}, \ldots, c_{t}^{\beta^{p-2}}\}$$
are bases of $A$. Without loss of generality we can assume that $b=b_1$ and $c=c_1$. Thus the map given by 
$$b_i \mapsto c_i \ \textrm{and} \  \alpha \mapsto \beta,$$
where $i = 1,\ldots, t$ extends to an automorphism of $G$. Hence all non-trivial elements of $A$ belong to the same orbit under the action of $\Aut(G)$, and all elements in $G \setminus A$ are in the same orbit under the action of $\Aut(G)$. The proof is complete. 
\end{proof}


\begin{thebibliography}{99}

\bibitem{BD18} R. Bastos and A.\,C. Dantas, \textit{FC-groups with finitely many automorphism orbits}, J. Algebra, {\bf 516} (2018) pp. 401--413.

\bibitem{BDG} R. Bastos, A.\,C. Dantas and M. Garonzi, \textit{Finite groups with six or seven automorphism orbits}, J. Group Theory, {\bf 21} (2017) pp.  945--954.

\bibitem{LM} T.\,J. Laffey and D. MacHale, {\it Automorphism orbits of finite groups}, J. Austral. Math. Soc. Ser. A, \textbf{40}(2) (1986) pp. 253--260.

\bibitem{Rob2} J. C. Lennox and D.\,J.\,S. Robinson, \textit{The Theory of Infinite Soluble Groups}, Clarendon Press, Oxford, 2004.

\bibitem{MS1} H. M\"aurer and M. Stroppel, \textit{Groups that are almost homogeneous}, Geom. Dedicata, \textbf{68} (1997) pp.  229--243.

\bibitem{Rob} D.\,J.\,S. Robinson, \textit{A course in the theory of groups}, 2nd edition, Springer-Verlag, New York, 1996.

 \bibitem{S2} M. Schwachh\"ofer and M. Stroppel, \textit{Finding representatives for the orbits under the automorphism group of a bounded
abelian group}, J. Algebra, {\bf 211} (1999) pp. 225--239.
 
\bibitem{S1} M. Stroppel, \textit{Locally compact groups with few orbits under automorphisms}, Top. Proc., {\bf 26}(2) (2002) pp. 819--842.

\end{thebibliography}
\end{document}